\documentclass[a4paper]{amsart}
\usepackage[T1]{fontenc}
\usepackage[british]{babel}
\usepackage{amssymb}
\usepackage{mathtools}
\usepackage[all]{xy}
\DeclareSymbolFont{bbold}{U}{bbold}{m}{n}
\DeclareSymbolFontAlphabet{\mathbbm}{bbold}
\newcommand{\M}{\mathcal{M}}
\newcommand{\N}{\mathcal{N}}
\newcommand{\B}{\mathbb{B}}
\newcommand{\C}{\mathbb{C}}
\newcommand{\D}{\mathbb{D}}
\newcommand{\bbot}{\mathbbm{0}}
\newcommand{\btop}{\mathbbm{1}}
\newcommand{\dd}{\mathfrak{d}}
\newcommand{\uu}{\mathfrak{u}}
\newcommand{\Mat}{\mathbb{M}}
\newcommand{\K}{\mathbb{K}}
\DeclareMathOperator{\dom}{dom}
\DeclareMathOperator{\cof}{cof}
\DeclareMathOperator{\add}{add}
\DeclareMathOperator{\non}{non}
\DeclareMathOperator{\Clop}{Clop}
\DeclarePairedDelimiter{\abs}{\lvert}{\rvert}
\DeclarePairedDelimiterX{\Set}[2]{\{}{\}}{\, #1 \mathclose{}\nonscript\;\delimsize\vert\nonscript\;\mathopen{} #2 \,}
\DeclarePairedDelimiterX{\Seq}[2]{\langle}{\rangle}{\, #1 \mathclose{}\nonscript\;\delimsize\vert\nonscript\;\mathopen{} #2 \,}
\DeclarePairedDelimiterXPP{\fin}[1]{}{[}{]}{^{<\aleph_0}}{#1}
\DeclarePairedDelimiterXPP{\eq}[2]{}{[}{]}{_{#2}}{#1}
\theoremstyle{plain}
\newtheorem{theorem}{Theorem}[section]
\newtheorem{proposition}[theorem]{Proposition}
\newtheorem{lemma}[theorem]{Lemma}
\newtheorem{corollary}[theorem]{Corollary}
\theoremstyle{definition}
\newtheorem{definition}[theorem]{Definition}
\newtheorem{question}[theorem]{Question}
\theoremstyle{remark}
\newtheorem{remark}[theorem]{Remark}
\begin{document}
\title{Combinatorics of ultrafilters on Cohen and random algebras}
\author{J\"org Brendle}
\author{Francesco Parente}
\thanks{The first author is partially supported by Grant-in-Aid for Scientific Research (C) 18K03398, Japan Society for the Promotion of Science. The second author is an International Research Fellow of the Japan Society for the Promotion of Science}
\address{Graduate School of System Informatics\\Kobe University\\1-1 Rokkodai-cho\\Nada-ku Kobe\\657-8501 Japan}
\begin{abstract}
We investigate the structure of ultrafilters on Boolean algebras in the framework of Tukey reducibility. In particular, this paper provides several techniques to construct ultrafilters which are not Tukey maximal. Furthermore, we connect this analysis with a cardinal invariant of Boolean algebras, the ultrafilter number, and prove consistency results concerning its possible values on Cohen and random algebras.
\end{abstract}
\maketitle

\section{Introduction}

Combinatorial properties of non-principal ultrafilters over $\omega$, that is, ultrafilters on the Boolean algebra $\mathcal{P}(\omega)/\mathrm{fin}$, have been extensively studied for the past half century. Key questions centred around the existence of ultrafilters with additional properties such as $P$-points, the Rudin-Keisler ordering, or cardinal invariants related to ultrafilters. Much less is known about ultrafilters on general Boolean algebras, although in recent years there has been considerable interest in their role in set theory and model theory. In particular, Malliaris and Shelah \cite{ms:dl} developed a new approach to Keisler's order based on morality of ultrafilters on Boolean algebras, whereas Goldstern, Kellner and Shelah \cite{gks:cm} used Boolean ultrapowers of forcing iterations to force all cardinals in Cicho\'n's diagram to be pairwise different.

In view of this situation, we investigate combinatorial aspects of ultrafilters on complete Boolean algebras, with a particular focus on the following two closely related topics:
\begin{itemize}
\item the existence of ultrafilters which are not Tukey maximal;
\item the ultrafilter number.
\end{itemize}
Our results are mostly (but not exclusively) about Cohen and random algebras.

Section \ref{section:due} begins by introducing the central notions of this paper. We then follow Kunen's framework \cite{kunen:rc} of index-invariant ideals to give examples of Boolean algebras on which every ultrafilter is Tukey maximal, answering an open question posed by Brown and Dobrinen \cite{bd:tukey}.

In Section \ref{section:tre} we discuss a cardinal invariant related to Tukey reducibility, namely the \emph{ultrafilter number} of Boolean algebras. In particular, we focus on the ultrafilter numbers of Cohen and random algebras, discussing their relation with previously studied cardinal invariants the continuum, and proving new consistency results about them via finite-support iterations.
	
The last two sections, on the other hand, deal with constructions of ultrafilters which are \emph{not} Tukey maximal. More specifically, in Section \ref{section:quattro} we construct non-maximal ultrafilters on Cohen algebras under the assumption that $\dd=2^{\aleph_0}$, whereas in Section \ref{section:cinque} we carry out such a construction on the random algebra, under the stronger assumption of $\Diamond$.

\section{Tukey-maximal ultrafilters}\label{section:due}

The notion of \emph{Tukey reducibility}, to which this section is dedicated, has its origin in the work of Tukey on convergence in general topology.

\begin{definition}[Tukey \cite{tukey:types}] Let $\bigl\langle D,\le^D\bigr\rangle$ and $\bigl\langle E,\le^E\bigr\rangle$ be directed sets. We define \[\bigl\langle D,\le^D\bigr\rangle\le_{\mathrm{T}}\bigl\langle E,\le^E\bigr\rangle\] if and only if there exist functions $f\colon D\to E$ and $g\colon E\to D$ such that for all $d\in D$ and $e\in E$
\[
f(d)\le^E e\implies d\le^D g(e).
\]
\end{definition}

The next theorem provides a general upper bound for directed sets of cardinality $\le\kappa$.

\begin{theorem}[{Tukey \cite[Theorem II-5.1]{tukey:types}}]\label{theorem:tukey} Let $\kappa$ be a cardinal and $\langle D,\le\rangle$ a directed set. If $\abs{D}\le\kappa$, then $\langle D,\le\rangle\le_{\mathrm{T}}\bigl\langle\fin*{\kappa},\subseteq\bigr\rangle$.
\end{theorem}

Given a directed set $\langle D,\le\rangle$, as usual we let $\cof(\langle D,\le\rangle)$ be the minimum cardinality of a cofinal subset of $D$. Furthermore, if $D$ has no maximum, let $\add(\langle D,\le\rangle)$ be the minimum cardinality of an unbounded subset of $D$.

\begin{proposition}[Schmidt \cite{schmidt:cof}]\label{proposition:schmidt} Let $\bigl\langle D,\le^D\bigr\rangle$ and $\bigl\langle E,\le^E\bigr\rangle$ be directed sets. If $D$ has no maximum and $\bigl\langle D,\le^D\bigr\rangle\le_{\mathrm{T}}\bigl\langle E,\le^E\bigr\rangle$, then
\[
\add(E)\le\add(D)\le\cof(D)\le\cof(E).
\]
\end{proposition}

Isbell \cite{isbell:cofinal} then initiated the study of Tukey reducibility of ultrafilters. Indeed, note that if $U$ is an ultrafilter on a Boolean algebra $\B$, then $\langle U,\ge\rangle$ is a directed set. Motivated by Theorem \ref{theorem:tukey}, we say that $U$ is \emph{Tukey maximal} if and only if $\bigl\langle\fin*{\B},\subseteq\bigr\rangle\le_{\mathrm{T}}\langle U,\ge\rangle$.

\begin{theorem}[{Isbell \cite[Theorem 5.4]{isbell:cofinal}}] For every infinite cardinal $\kappa$, there exists a Tukey-maximal ultrafilter over $\kappa$.
\end{theorem}

The above theorem stimulated a fruitful line of research, which is nicely surveyed by Dobrinen \cite{dobrinen:tukey} and constitutes the main motivation for this work. The following criterion is particularly useful to determine whether an ultrafilter is Tukey maximal. We omit the proof, which is the same as Dobrinen and Todor\v{c}evi\'c \cite[Fact 12]{dt:tukey}.

\begin{proposition}\label{proposition:criterion} Let $\kappa$ be an infinite cardinal. For an ultrafilter $U$ on a Boolean algebra $\B$, the following conditions are equivalent:
\begin{itemize}
\item $\bigl\langle\fin*{\kappa},\subseteq\bigr\rangle\le_{\mathrm{T}}\langle U,\ge\rangle$;
\item there exists a subset $X\subseteq U$ with $\abs{X}=\kappa$ such that every infinite $Y\subseteq X$ is unbounded in $U$.
\end{itemize}
\end{proposition}

At this point, it is appropriate to discuss the connection with the Rudin-Keisler ordering and ultrafilters over $\omega$. Recall that, whenever $U$ and $V$ are ultrafilters over $\omega$, we have $U\le_\mathrm{RK} V$ if there exists a function $f\colon\omega\to\omega$ such that for all $X\subseteq\omega$
\[
X\in U\iff f^{-1}[X]\in V.
\]
Now, as noted by Dobrinen and Todor\v{c}evi\'c \cite[Fact 1]{dt:tukey}, if $U\le_\mathrm{RK} V$ then $\langle U,\supseteq\rangle\le_\mathrm{T}\langle V,\supseteq\rangle$. We aim to show that this implication is also true, in the Boolean-algebraic context, for the generalized Rudin-Keisler ordering introduced by Murakami.

\begin{definition}[Murakami \cite{murakami:rk}] Let $\B$ and $\C$ be complete Boolean algebras. For ultrafilters $U\subset\B$ and $V\subset\C$, we define $U\le_\mathrm{RK} V$ if there exist $v\in V$ and a complete homomorphism $f\colon\B\to\C\mathbin{\upharpoonright} v$ such that $U=f^{-1}[V]$.
\end{definition}

\begin{proposition}\label{proposition:rk} Let $\B$ and $\C$ be complete Boolean algebras, with ultrafilters $U\subset\B$ and $V\subset\C$. If $U\le_\mathrm{RK}V$, then $\langle U,\ge\rangle\le_\mathrm{T}\langle V,\ge\rangle$.
\end{proposition}
\begin{proof} Suppose there exist $v\in V$ and a complete homomorphism $f\colon\B\to\C\mathbin{\upharpoonright} v$ such that $U=f^{-1}[V]$. Let us define
\[
\begin{split}
g\colon\C &\longrightarrow\B \\
c &\longmapsto\bigwedge\Set{b\in\B}{c\le f(b)}
\end{split}
\]
and note that $c\le f(b)$ implies $g(c)\le b$. Furthermore, observe that if $c\in V$ then, by completeness of the homomorphism $f$,
\[
f(g(c))=f\Bigl(\bigwedge\Set{b\in\B}{c\le f(b)}\Bigr)=\bigwedge\Set{f(b)}{c\le f(b)}\ge c,
\]
hence $f(g(c))\in V$, which means $g(c)\in U$. In conclusion, the maps $f$ and $g$ witness that $\langle U,\ge\rangle\le_\mathrm{T}\langle V,\ge\rangle$.
\end{proof}

The following is a straightforward corollary which clarifies the relationship to ultrafilters over $\omega$.

\begin{corollary}\label{corollary:rkomega} Let $\B$ be a complete Boolean algebra. For an ultrafilter $V\subset\B$, the following conditions are equivalent:
\begin{enumerate}
\item\label{corollary:rkomegauno} $V$ is not $\sigma$-complete;
\item\label{corollary:rkomegadue} there exists a non-principal ultrafilter $U$ over $\omega$ such that $U\le_\mathrm{RK}V$;
\item\label{corollary:rkomegatre} there exists a non-principal ultrafilter $U$ over $\omega$ such that $\langle U,\supseteq\rangle\le_\mathrm{T}\langle V,\ge\rangle$.
\end{enumerate}
\end{corollary}
\begin{proof} $(\ref{corollary:rkomegauno}\Longrightarrow\ref{corollary:rkomegadue})$ Assuming $V$ is not $\sigma$-complete, there exists a maximal antichain $\Set{a_i}{i<\omega}$ in $\B$ such that for all $i<\omega$, $a_i\notin V$. The function
\[
\begin{split}
f\colon\mathcal{P}(\omega) &\longrightarrow \B \\
X &\longmapsto \bigvee\Set{a_i}{i\in X}
\end{split}
\]
is a complete homomorphism, hence $U=f^{-1}[V]$ is the desired ultrafilter.

$(\ref{corollary:rkomegadue}\Longrightarrow\ref{corollary:rkomegatre})$ By Proposition \ref{proposition:rk}.

$(\ref{corollary:rkomegatre}\Longrightarrow\ref{corollary:rkomegauno})$ Suppose $U$ is a non-principal ultrafilter over $\omega$ such that $\langle U,\supseteq\rangle\le_\mathrm{T}\langle V,\ge\rangle$. By Proposition \ref{proposition:schmidt} we have $\add(V)\le\add(U)=\aleph_0$, implying $V$ is not $\sigma$-complete.
\end{proof}

In the remainder of this section, we introduce the framework of Kunen \cite{kunen:rc} and show that certain Boolean algebras have only ultrafilters which are Tukey maximal.

\begin{definition} A $\sigma$-ideal $I$ over $^\omega2$ is \emph{index invariant} if for every injective function $\Delta\colon\omega\to\omega$ and $X\subseteq{^\omega}2$ we have
\[
X\in I\iff\Set{f\in{^\omega}2}{f\circ\Delta\in X}\in I.
\]

Let $I$ be an index-invariant $\sigma$-ideal over $^\omega2$ and $\alpha$ an infinite ordinal. We define $I(\alpha)\subset\mathcal{P}({^\alpha}2)$ as follows: $X\in I(\alpha)$ if and only if there exist an injective function $\Delta\colon\omega\to\alpha$ and a set $Y\in I$ such that for all $f\in X$, $f\circ\Delta\in Y$.
\end{definition}

Prototypical examples of index-invariant $\sigma$-ideals are of course the meagre ideal $\M$ and the null ideal $\N$, which will be discussed in greater detail from Section \ref{section:tre}.

\begin{lemma}[{Kunen \cite[Lemma 1.5]{kunen:rc}}]\label{lemma:kunen} Let $I$ be an index-invariant $\sigma$-ideal over $^\omega2$. For every infinite ordinal $\alpha$, $I(\alpha)$ is indeed a $\sigma$-ideal over $^\alpha2$.

Furthermore, $I(\alpha)$ is also ``index invariant'' in the sense that, for every injective function $\Gamma\colon\alpha\to\beta$ and $X\subseteq{^\alpha}2$,
\[
X\in I(\alpha)\iff\Set*{f\in{^\beta}2}{f\circ\Gamma\in X}\in I(\beta).
\]
\end{lemma}

Let $\Clop(^\alpha2)$ be the Boolean algebra of clopen subsets of the Cantor space $^\alpha2$, and $\mathcal{B}(^\alpha2)$ be the $\sigma$-algebra generated by $\Clop(^\alpha2)$; note that $\abs{\Clop(^\alpha2)}=\abs{\alpha}$ whereas $\abs{\mathcal{B}(^\alpha2)}=\abs{\alpha}^{\aleph_0}$. If $I$ is an index-invariant $\sigma$-ideal over $^\omega2$, let
\begin{equation}\label{eq:quot}
\B(I,\alpha)=\mathcal{B}(^\alpha2)/I(\alpha)
\end{equation}
be the quotient $\sigma$-algebra.

\begin{remark}\label{remark:induce} By Lemma \ref{lemma:kunen}, every injective function $\Gamma\colon\alpha\to\beta$ induces a $\sigma$-complete embedding $\Gamma_*\colon\B(I,\alpha)\to\B(I,\beta)$ such that if $X\in\mathcal{B}(^\alpha2)$ then
\[\Gamma_*\bigl(\eq*{X}{I(\alpha)}\bigr)=\eq*{\Set*{f\in{^\beta}2}{f\circ\Gamma\in X}}{I(\beta)}.
\]
\end{remark}

\begin{theorem}\label{theorem:tukeymax} Let $I$ be an index-invariant $\sigma$-ideal over $^\omega2$, containing all singletons. For every infinite cardinal $\kappa$ and every ultrafilter $U$ on $\B(I,\kappa)$, we have ${\bigl\langle\fin*{\kappa},\subseteq\bigr\rangle}\le_{\mathrm{T}}\langle U,\ge\rangle$.
\end{theorem}
\begin{proof} Let $U$ be an ultrafilter on $\B(I,\kappa)$. We define a function $x\colon\kappa\times 2\to\Clop(^\kappa2)$ as follows: for every $\alpha<\kappa$,
\[
x(\alpha,0)=\Set{f\in{^\kappa}2}{f(\alpha)=0}\quad\text{and}\quad x(\alpha,1)=\Set{f\in{^\kappa}2}{f(\alpha)=1}.
\]
As $U$ is an ultrafilter, there exists a function $g\colon\kappa\to 2$ such that for all $\alpha<\kappa$, $\eq*{x\bigl(\alpha,g(\alpha)\bigr)}{I(\kappa)}\in U$. Let
\[
X=\Set*{\eq*{x\bigl(\alpha,g(\alpha)\bigr)}{I(\kappa)}}{\alpha<\kappa};
\]
it is sufficient to show that whenever $Y\subseteq X$ is infinite we have $\bigwedge Y\notin U$, then conclude using Proposition \ref{proposition:criterion}.

Let $\Delta\colon\omega\to\kappa$ be an arbitrary injective function. Since we are assuming $I$ contains all singletons, we have $\{g\circ\Delta\}\in I$. Therefore, by definition,
\[
\bigcap_{n<\omega}x\bigl(\Delta(n),g(\Delta(n))\bigr)=\Set*{f\in{^\kappa}2}{f\circ\Delta = g\circ\Delta}=\Set*{f\in{^\kappa}2}{f\circ\Delta\in\{g\circ\Delta\}}\in I(\kappa).
\]
Now using the fact that $I(\kappa)$ is a $\sigma$-ideal,
\[
\bigwedge_{n<\omega}\eq*{x\bigl(\Delta(n),g(\Delta(n))\bigr)}{I(\kappa)}=\eq*{\bigcap_{n<\omega}x\bigl(\Delta(n),g(\Delta(n))\bigr)}{I(\kappa)}=\bbot\notin U,
\]
as desired.
\end{proof}

Brown and Dobrinen \cite[Question 4.2]{bd:tukey} asked: if $\B$ is an infinite Boolean algebra such that all ultrafilters on $\B$ are Tukey maximal, is $\B$ necessarily a free algebra? Theorem \ref{theorem:tukeymax} provides a negative answer to this question. Indeed it implies that, if $\kappa$ is a cardinal satisfying $\kappa^{\aleph_0}=\kappa$ and $I$ is, for instance, the meagre (or null) ideal, then every ultrafilter on $\B(I,\kappa)$ is Tukey maximal. On the other hand, $\B(I,\kappa)$ is a $\sigma$-algebra, hence not free by the Gaifman-Hales theorem \cite{gaifman:free,hales:free}.

\section{The ultrafilter number}\label{section:tre}

This section is dedicated to the study of a cardinal invariant, the ultrafilter number, which is closely related to Tukey reducibility. We shall carry out this study for Boolean algebras of the form $\B(I,\kappa)$, focusing in particular on the situation where $\kappa$ is either $\omega$ or $\omega_1$, so that the ultrafilter numbers $\uu(\B(I,\omega))$ and $\uu(\B(I,\omega_1))$ are indeed two cardinal invariants of the continuum. We begin with standard definitions.

Let $\B$ be an infinite Boolean algebra. We let
\[
\pi(\B)=\min\Set{\abs{\D}}{\D\text{ is a dense subalgebra of }\B}
\]
be the \emph{$\pi$-weight} of $\B$. Furthermore, we define the \emph{ultrafilter number} of $\B$ as
\[
\uu(\B)=\min\Set{\cof(\langle U,\ge\rangle)}{U\subset\B\text{ is a non-principal ultrafilter}};
\]
for simplicity of notation, let also $\uu=\uu(\mathcal{P}(\omega)/\mathrm{fin})$. Finally, as usual $\dd$ denotes the \emph{dominating number}.

The main motivation to consider the ultrafilter number in this paper is given by the following observation, which follows directly from Proposition \ref{proposition:criterion}.

\begin{remark} Let $\B$ be an infinite Boolean algebra. If $\uu(\B)<\abs{\B}$, then there exists a non-principal ultrafilter on $\B$ which is not Tukey maximal.
\end{remark}

Let $\N$ be the ideal over $^\omega2$ consisting of sets which are null with respect to the standard product measure $\mu$. The cardinal $\non(\N)$, which is the least size of a set which is not null, will play a role in our discussion. We also assume some familiarity with the meagre ideal $\M$, consisting of sets which are countable unions of nowhere dense sets.

As $\M$ and $\N$ are both index-invariant $\sigma$-ideals, containing all singletons, we shall consider the corresponding quotient algebras as in \eqref{eq:quot}. In accord with usual notation, for an infinite ordinal $\alpha$ let $\C_\alpha=\B(\M,\alpha)$ be the \emph{Cohen algebra} and $\B_\alpha=\B(\N,\alpha)$ be the \emph{Random algebra}. 

\begin{proposition}\label{proposition:dense} Let $\alpha$ be an infinite ordinal; then both $\C_\alpha$ and $\B_\alpha$ are c.c.c.\ complete Boolean algebras. Furthermore,
\begin{enumerate}
\item\label{denseuno} for every $B\in\mathcal{B}(^\alpha2)\setminus\M(\alpha)$ there exists a non-empty $C\in\Clop(^\alpha2)$ such that $C\setminus B\in\M(\alpha)$;
\item\label{densedue} for every $B\in\mathcal{B}(^\alpha2)\setminus\N(\alpha)$ and every $\varepsilon>0$ there exists a non-empty $C\in\Clop(^\alpha2)$ such that $\mu(B\cap C)\ge(1-\varepsilon)\mu(C)$.
\end{enumerate}
\end{proposition}

We omit the proof of this well-known fact, but we remark that \eqref{denseuno} can be found in Sikorski \cite[Theorem 35.1]{sikorski:boolean}, while \eqref{densedue} follows from Lebesgue's density theorem.

We are ready to undertake the study of $\uu(\C_\omega)$, $\uu(\C_{\omega_1})$, $\uu(\B_\omega)$ and $\uu(\B_{\omega_1})$, beginning with a theorem which summarizes the relations between those four cardinals and other cardinal invariants of the continuum.

\begin{theorem}\label{theorem:hasse} The inequalities in the following Hasse diagram hold:
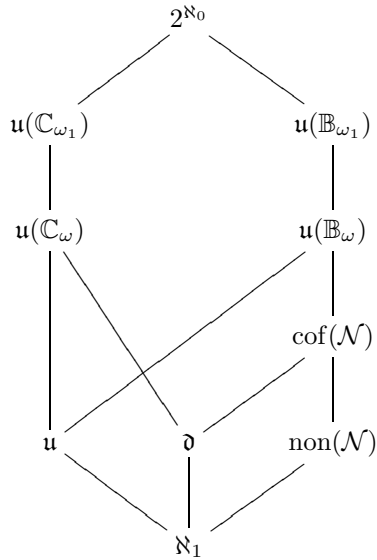
\[
\xymatrix{
& {2^{\aleph_0}} & \\
{\uu(\C_{\omega_1})}\ar@{-}[ur] & & {\uu(\B_{\omega_1})}\ar@{-}[ul] \\
{\uu(\C_\omega)}\ar@{-}[u] & & {\uu(\B_\omega)}\ar@{-}[u] \\
& & {\cof(\N)}\ar@{-}[u] \\
{\uu}\ar@{-}[uu]\ar@{-}[uurr] & {\dd}\ar@{-}[uul]\ar@{-}[ur] & {\non(\N)}\ar@{-}[u] \\
& {\aleph_1}\ar@{-}[ul]\ar@{-}[u]\ar@{-}[ur]}
\]
\end{theorem}
\begin{proof} The upper bound $\uu(\C_{\omega_1})+\uu(\B_{\omega_1})\le 2^{\aleph_0}$ follows directly from the observation that the Boolean algebras $\C_{\omega_1}$ and $\B_{\omega_1}$ have both cardinality $2^{\aleph_0}$.

As for $\uu(\C_\omega)\le\uu(\C_{\omega_1})$, Remark \ref{remark:induce} gives a $\sigma$-complete embedding $f\colon\C_\omega\to\C_{\omega_1}$ which, due to the fact that $\C_\omega$ is c.c.c., is in fact a complete embedding. Consequently, for every ultrafilter $U$ on $\C_{\omega_1}$ we have $\cof\bigl(f^{-1}[U]\bigr)\le\cof(U)$ by Proposition \ref{proposition:schmidt} and Proposition \ref{proposition:rk}. The inequality $\uu(\B_\omega)\le\uu(\B_{\omega_1})$ is proved analogously.

To prove $\uu\le\uu(\C_\omega)$, it is sufficient to note that no ultrafilter on $\C_\omega$ is $\sigma$-complete, then apply Proposition \ref{proposition:schmidt} and Corollary \ref{corollary:rkomega}. Similarly for $\uu\le\uu(\B_\omega)$.

We now prove the inequality $\dd\le\uu(\C_\omega)$. Let $U$ be an ultrafilter on $\C_\omega$; since $U$ is not $\sigma$-complete, there exists a maximal antichain $\Set{a_i}{i<\omega}$ such that for all $i<\omega$, $a_i\notin U$. Let $X\subset U$ be such that $\abs{X}<\dd$, we aim to prove that $X$ is not cofinal in $\langle U,\ge\rangle$. Whenever $s\in{^{<\omega}2}$, we let
\[
N_s=\Set{f\in{^\omega2}}{s\subset f}
\]
be the corresponding element of $\Clop(^\omega2)$. By point \eqref{denseuno} of Proposition \ref{proposition:dense}, for every $b\in\C_\omega\setminus\{\bbot\}$ there exists some $s\in{^{<\omega}2}$ such that $\eq{N_s}{\M}\le b$. Using this fact, for each $x\in X$ we may define a function $f_x\colon\omega\to\omega$ as follows. Let $i<\omega$; as every element of $U$ meets infinitely many members of $\Set{a_i}{i<\omega}$, let $j_i$ be the least $j\ge i$ such that $a_j\wedge x>\bbot$. Then we define
\[
f_x(i)=\min\Set{n<\omega}{\text{there exists }s\in{^n2}\text{ such that }\eq{N_s}{\M}\le a_{j_i}\wedge x}.
\]
Since $\abs{X}<\dd$, there exists an increasing function $g\colon\omega\to\omega$ such that for all $x\in X$ we have $g\nleq f_x$. Now define
\[
u=\bigvee_{i<\omega}\bigl(a_i\wedge\eq{\Set{f\in{^\omega2}}{f(g(i))=0}}{\M}\bigr)
\]
and assume $u\in U$; the other case when $\neg u\in U$ is completely analogous. Towards a contradiction suppose now $X$ is cofinal in $\langle U,\ge\rangle$, this implies the existence of some $x\in X$ such that $x\le u$. Since $g$ is not dominated by $f_x$, we can find some $i<\omega$ such that $f_x(i)<g(i)$; then, by definition of $f_x$ there exists some $s\colon f_x(i)\to 2$ such that $\eq{N_s}{\M}\le a_{j_i}\wedge x$. Putting everything together, we deduce
\[
\eq{N_s}{\M}\le a_{j_i}\wedge x\le a_{j_i}\wedge u\le\eq{\Set{f\in{^\omega2}}{f(g(j_i))=0}}{\M},
\]
which is a contradiction because $f_x(i)<g(j_i)$. This shows that $X$ cannot be cofinal in the ultrafilter and completes the proof of $\dd\le\uu(\C_\omega)$.

The fact that $\cof(\N)\le\uu(\B_\omega)$ is a consequence of the work of Burke \cite{burke:dense}. More precisely, in Case 1 of \cite[Theorem 1]{burke:dense}, he showed the following: if $X\subseteq\B_\omega\setminus\{\bbot\}$ has the property that for all $b\in\B_\omega$ there exists $x\in X$ such that either $x\le b$ or $x\wedge b=\bbot$, then there exists a cofinal subset $Y\subseteq\N$ such that $\abs{X}=\abs{Y}$. Therefore, the inequality $\cof(\N)\le\uu(\B_\omega)$ follows from the observation that, whenever $U$ is an ultrafilter on $\B_\omega$ and $X\subseteq U$ is cofinal, clearly $X$ satisfies the assumption of Burke's result.

The other inequalities in the diagram are well known.
\end{proof}

Regarding the inequalities in the above theorem, we still do not know whether it is true that $\uu(\C_\omega)=\uu(\C_{\omega_1})$ and $\uu(\B_\omega)=\uu(\B_{\omega_1})$ in $\mathrm{ZFC}$. We now present two consistency results, for which we shall assume some familiarity with finite-support iteration of c.c.c.\ forcing notions.

\begin{definition} A notion of forcing $\mathbb{P}$ is \emph{$\sigma$-centred} if $\mathbb{P}=\bigcup_{n<\omega}P_n$ where, for all $n<\omega$, every finite subset of $P_n$ has a lower bound in $\mathbb{P}$.
\end{definition}

It is well known that a finite-support iteration of $\sigma$-centred forcing notions is also $\sigma$-centred; see Tall \cite{tall:centred} for further details on this topic.

\begin{theorem}\label{theorem:c1} It is consistent that $\uu(\C_{\omega_1})<\non(\N)$.
\end{theorem}

Towards a proof, we first introduce a notion of forcing inspired by Mathias \cite[Definition 1.0]{mathias:happy} and discuss its basic features. Let $\Set{a_i}{i<\omega}$ be a maximal antichain in $\C_\omega$, which will be fixed throughout the construction. For $\alpha<\omega_1$, let also $\D_\alpha\subseteq\C_\alpha$ be a countable dense subalgebra. As previously observed, whenever $\alpha\le\beta$ we have naturally a complete embedding $\C_\alpha\longrightarrow\C_\beta$, hence without loss of generality we may represent $\C_{\omega_1}$ as an increasing union of subalgebras: $\C_{\omega_1}=\bigcup_{\alpha<\omega_1}\C_\alpha$.

For every infinite $\alpha<\omega_1$ and every ultrafilter $U$ on $\C_\alpha$ with the property that for all $i<\omega$, $a_i\notin U$, we define a notion of forcing $\Mat(U,\C_\alpha)$. Forcing conditions are pairs $p=\langle s_p,S_p\rangle$, where:
\begin{itemize}
\item $s_p\in\C_\alpha$;
\item for all $i<\omega$, $a_i\wedge s_p\in\D_\alpha$;
\item for all but finitely many $i<\omega$, $a_i\wedge s_p=\bbot$;
\item $S_p\in U$.
\end{itemize}
The ordering is defined as follows: for $p,q\in\Mat(U,\C_\alpha)$, by definition
\[
q\le_\alpha p\iff s_p\le s_q\text{ and }S_q\le S_p\text{ and }s_q\wedge\neg s_p\le S_p.
\]
It is now straightforward to check that $\langle\Mat(U,\C_\alpha),\le_\alpha\rangle$ is a partially ordered set.

\begin{lemma} The notion of forcing $\Mat(U,\C_\alpha)$ is $\sigma$-centred.
\end{lemma}
\begin{proof} As the subalgebra $\D_\alpha$ is countable, there are only countably many $s\in\C_\alpha$ such that $\Set{a_i\wedge s}{i<\omega}\subseteq\D_\alpha$ and the set $\Set{i<\omega}{a_i\wedge s>\bbot}$ is finite. For every such $s$, let $M_s=\Set{p\in\Mat(U,\C_\alpha)}{s_p=s}$, so that $\Mat(U,\C_\alpha)$ is the countable union of the $M_s$. Furthermore, finitely many conditions $p_1,\dots,p_n\in M_s$ are compatible, as witnessed by the stronger condition $\langle s,S_{p_1}\wedge\dots\wedge S_{p_n}\rangle$.
\end{proof}

For an $\Mat(U,\C_\alpha)$-generic filter $G$, let
\[
g=\bigvee\Set{s_p}{p\in G};
\]
we then have the following:
\begin{enumerate}
\item\label{pseudo1} $\Set{i<\omega}{a_i\wedge g>\bbot}$ is infinite, but
\item\label{pseudo2} for each $S\in U$, $\Set{i<\omega}{a_i\wedge g\wedge\neg S>\bbot}$ is finite.
\end{enumerate}
The proof relies on standard density arguments. Indeed, to establish Property \eqref{pseudo1} it is sufficient to prove that for every $n<\omega$ the set
\[
D_n=\Set{p\in\Mat(U,\C_\alpha)}{\abs{\Set{i<\omega}{a_i\wedge s_p>\bbot}}\ge n}
\]
is dense in $\Mat(U,\C_\alpha)$. So let $p$ be an arbitrary condition; since $S_p\in U$, we have that $I=\Set{i<\omega}{a_i\wedge S_p>\bbot}$ is infinite, so it is possible to find a strictly increasing sequence $\Seq{i_k}{k<n}$ within $I$. By density of $\D_\alpha$, for every $k<n$ let $d_k\in\D_\alpha$ be such that $d_k\le a_{i_k}\wedge S_p$. Then it is clear that $\langle s_p\vee d_0\vee\dots\vee d_{n-1},S_p\rangle$ is a condition in $D_n$ which is stronger than $p$.

As for Property \eqref{pseudo2}, given $S\in U$ the set
\[
\Set{p\in\Mat(U,\C_\alpha)}{S_p\le S}
\]
is dense. Therefore, there exists $p\in G$ such that $S_p\le S$. For every $q\in G$, there exists a condition $r$ which is stronger than both $p$ and $q$, hence
\[
s_q\wedge\neg s_p\le s_r\wedge\neg s_p\le S_p\le S.
\]
As $q\in G$ is arbitrary, we deduce $g\wedge\neg S\le s_p$ which concludes the proof.

In the forcing extension $V[G]$, the set $U\cup\{g\}$ has the finite intersection property. Furthermore, if $U'$ is any ultrafilter on $\C_{\alpha+1}$ extending $U\cup\{g\}$, then $U'$ also has the property that for each $i<\omega$, $a_i\notin U'$. We now turn to the main lemma towards the proof of Theorem \ref{theorem:c1}.

\begin{lemma}\label{lemma:main1} It is consistent that $\uu(\C_{\omega_1})<2^{\aleph_0}$.
\end{lemma}
\begin{proof} Let us assume $V$ satisfies $\aleph_1<2^{\aleph_0}$. We proceed with a finite-support iteration $\Seq{\mathbb{P}_\alpha,\dot{\mathbb{Q}}_\alpha}{\alpha<\omega_1}$ of the forcing introduced above; alongside the iteration, for each $\alpha<\omega_1$ we construct a $\mathbb{P}_\alpha$-name $\dot{U}_\alpha$ for an ultrafilter on $\C_{\omega+\alpha}$. As the iteration is $\sigma$-centred, in particular the c.c.c.\ will guarantee that cardinals are preserved in the forcing extension.

First of all, let $\mathbb{P}_0=\{\emptyset\}$ and let $\dot{U}_0$ be the name for any ground-model ultrafilter $U_0$ on $\C_\omega$ with the property that for all $i<\omega$, $a_i\notin U_0$. For $\alpha<\omega_1$, suppose $\mathbb{P}_\alpha$ and $\dot{U}_\alpha$ are given; we let
\[
\mathbb{P}_{\alpha+1}=\mathbb{P}_\alpha*\dot{\Mat}(U_\alpha,\C_{\omega+\alpha})
\]
and $\dot{g}_\alpha$ be a $\mathbb{P}_{\alpha+1}$-name for the element of $\C_{\omega+\alpha}$ added generically at this stage. We then take $\dot{U}_{\alpha+1}$ to be the name for an ultrafilter on $\C_{\omega+\alpha+1}$ extending $\dot{U}_\alpha\cup\{\dot{g}_\alpha\}$. Finally, at limit ordinals $\delta<\omega_1$, we simply let $\dot{U}_\delta$ be a name for an ultrafilter on $\C_{\omega+\delta}$ extending $\bigcup_{\alpha<\delta}\dot{U}_\alpha$.

In the resulting forcing extension $V[G]$, let $U=\bigcup_{\alpha<\omega_1}U_\alpha$. We claim that $U$ is an ultrafilter on $\C_{\omega_1}$; indeed, if $b\in\C_{\omega_1}$ then there exists some $\alpha<\omega_1$ such that $b\in\C_\alpha$. Hence, at some later stage of the iteration, either $b$ or $\neg b$ is decided to be in $U$. In conclusion, $U$ is an ultrafilter on $\C_{\omega_1}$ such that
\[
\Set*{g_\alpha\wedge\bigvee\Set{a_i}{i\ge n}}{\alpha<\omega_1,\ n<\omega}
\]
is cofinal in $\langle U,\ge\rangle$, thus witnessing the fact that $\aleph_1=\uu(\C_{\omega_1})<2^{\aleph_0}$ in $V[G]$.
\end{proof}

\begin{proof}[Proof of Theorem \ref{theorem:c1}] Let us assume $V$ satisfies Martin's Axiom \cite{martinsolovay:internal} and consider the forcing of Lemma \ref{lemma:main1}, which resulted in a model of $\uu(\C_{\omega_1})<2^{\aleph_0}$. The forcing is $\sigma$-centred, being an iteration of $\sigma$-centred forcings, and it is known that such extension will also satisfy $\non(\N)=2^{\aleph_0}$. Indeed, this is explained in detail in the proof of \cite[Corollary 41]{bbbhhl:rearrangement}, which is enough to conclude the proof of our theorem.
\end{proof}

We have thus established Theorem \ref{theorem:c1} which, together with the inequality $\non(\N)\le\uu(\B_\omega)$ discussed in Theorem \ref{theorem:hasse}, also yields the consistency of $\uu(\C_{\omega_1})<\uu(\B_\omega)$. We now turn to the second consistency result.

\begin{theorem}\label{theorem:c2} It is consistent that $\uu(\B_{\omega_1})<2^{\aleph_0}$.
\end{theorem}

Our proof employs the finite-support iteration of a forcing notion originally introduced by Kunen \cite{kunen:points}. For every infinite $\alpha<\omega_1$ and every ultrafilter $U$ on $\B_\alpha$, let $\K(U,\B_\alpha)$ be the notion of forcing described as follows. Conditions are functions $p\colon\omega\to U$, where
\begin{itemize}
\item for all $n<\omega$, $p(n+1)\le p(n)$;
\item $p$ is eventually constant;
\item for all $n<\omega$, $\mu(p(n))>\frac{1}{2^{n+1}}$.
\end{itemize}
The order relation is pointwise: for $p,q\in\K(U,\B_\alpha)$, by definition
\[
q\le_\alpha p\iff\forall n(q(n)\le p(n)).
\]

The following lemma, although originally stated as a consequence of Martin's Axiom, has a straightforward translation to the present context.
 
\begin{lemma}[{Kunen \cite[Lemma 4.3.1]{kunen:points}}] The notion of forcing $\K(U,\B_\alpha)$ is c.c.c.\ Furthermore, if $G$ is a $\K(U,\B_\alpha)$-generic filter and $g\colon\omega\to\B_\alpha$ is defined by
\[
g(n)=\bigwedge\Set{p(n)}{p\in G},
\]
then the following hold:
\begin{itemize}
\item for all $n<\omega$, $g(n+1)\le g(n)$;
\item for all $n<\omega$, $\mu(g(n))\ge\frac{1}{2^{n+1}}$;
\item for every $u\in U$ there exists $n<\omega$ such that $g(n)\le u$.
\end{itemize}
\end{lemma}

Consequently, in the forcing extension $V[G]$, the set $U\cup\Set{g(n)}{n<\omega}$ has the finite intersection property.

\begin{proof}[Proof of Theorem \ref{theorem:c2}] Let us assume $V$ satisfies $\aleph_1<2^{\aleph_0}$. As in the proof of Theorem \ref{theorem:c1}, we proceed with a finite-support iteration $\Seq{\mathbb{P}_\alpha,\dot{\mathbb{Q}}_\alpha}{\alpha<\omega_1}$ of the c.c.c.\ forcing just introduced; alongside the iteration, for each $\alpha<\omega_1$ we construct a $\mathbb{P}_\alpha$-name $\dot{U}_\alpha$ for an ultrafilter on $\B_{\omega+\alpha}$.

The base case of the iteration presents no complications. For $\alpha<\omega_1$, suppose $\mathbb{P}_\alpha$ and $\dot{U}_\alpha$ are given; we let
\[
\mathbb{P}_{\alpha+1}=\mathbb{P}_\alpha*\dot{\K}(U_\alpha,\B_{\omega+\alpha})
\]
and $\dot{g}_\alpha$ be a $\mathbb{P}_{\alpha+1}$-name for the function $\omega\longrightarrow\B_{\omega+\alpha}$ added generically at this stage. We then take $\dot{U}_{\alpha+1}$ to be the name for an ultrafilter on $\B_{\omega+\alpha+1}$ extending $\dot{U}_\alpha\cup\Set{\dot{g}_\alpha(n)}{n<\omega}$. Finally, at countable limit ordinals, we proceed as usual by taking unions.

In the resulting forcing extension $V[G]$, let $U=\bigcup_{\alpha<\omega_1}U_\alpha$. Then, $U$ is an ultrafilter on $\B_{\omega_1}$ such that
\[
\Set{g_\alpha(n)}{\alpha<\omega_1,\ n<\omega}
\]
is cofinal in $\langle U,\ge\rangle$, thus witnessing the fact that $\aleph_1=\uu(\B_{\omega_1})<2^{\aleph_0}$ in $V[G]$.
\end{proof}

The following question remains open:

\begin{question} Is it consistent that $\uu(\B_\omega)<\uu(\C_\omega)$?
\end{question}

\section{Non-maximal ultrafilters on Cohen algebras}\label{section:quattro}

Usual constructions of ultrafilters over $\omega$ which are not Tukey maximal, as in Dobrinen and Todor\v{c}evi\'c \cite[Section 3]{dt:tukey}, rely on the existence of $P$-points.

\begin{definition}[Gillman and Henriksen \cite{gh:ppoint}] An ultrafilter $U$ over $\omega$ is a \emph{$P$-point} if whenever $\Set{X_n}{n<\omega}\subset U$, there exists $Y\in U$ such that for each $n<\omega$, $Y\setminus X_n$ is finite.
\end{definition}

In this section, we shall make use of a related combinatorial notion to obtain non-maximal ultrafilters on some complete c.c.c.\ Boolean algebras, under the assumption that $\dd=2^{\aleph_0}$.

\begin{definition}[{Star\'y \cite[Definition 3.1]{stary:coherent}}] Let $\B$ be a complete c.c.c.\ Boolean algebra. An ultrafilter $U$ on $\B$ is a \emph{coherent $P$-ultrafilter} if for every maximal antichain $\Set{p_i}{i<\omega}$ in $\B$, the set
\[
\Set*{X\subseteq\omega}{\bigvee\Set{p_i}{i\in X}\in U}
\]
is a $P$-point ultrafilter over $\omega$.
\end{definition}

We begin with a reformulation which will be useful for the proof of Theorem \ref{theorem:weight}.

\begin{lemma}\label{lemma:char} Let $\B$ be a complete c.c.c.\ Boolean algebra. For an ultrafilter $U$ on $\B$, the following conditions are equivalent:
\begin{enumerate}
\item\label{lemma:charuno} $U$ is a coherent $P$-ultrafilter;
\item\label{lemma:chardue} for every maximal antichain $\Set{a_i}{i<\omega}$ in $\B$ and every $\Set{x_n}{n<\omega}\subseteq U$, there exists $y\in U$ such that for each $n<\omega$ the set $\Set*{i<\omega}{a_i\wedge y\wedge\neg x_n>\bbot}$ is finite.
\end{enumerate}
\end{lemma}
\begin{proof} $(\ref{lemma:charuno}\Longrightarrow\ref{lemma:chardue})$ Suppose $U$ is a coherent $P$-ultrafilter; let a maximal antichain $\Set{a_i}{i<\omega}$ and a subset $\Set{x_n}{n<\omega}\subseteq U$ be given. We aim to find first a maximal antichain $\Set{p_i}{i<\omega}$ such that
\begin{itemize}
\item for all $i<\omega$ there exists $j<\omega$ such that $p_i\le a_j$;
\item for every $n<\omega$ the set $\Set{i<\omega}{p_i\wedge x_n>\bbot\text{ and }p_i\wedge\neg x_n>\bbot}$ is finite.
\end{itemize}
To do so, it is convenient to use the following notation: for $n<\omega$, let
\[
x^1_n=x_n\quad\text{ and }x^0_n=\neg x_n.
\]
Now, for $s\in{^{<\omega}2}$ we define
\[
b_s=a_{\dom(s)}\wedge\bigwedge\Set*{x^{s(n)}_n}{n<\dom(s)}.
\]

First, it is clear that for all $s,t\in{^{<\omega}2}$, if $b_s\wedge b_t>\bbot$ then $s=t$. Secondly, by finite distributivity we have for each $i<\omega$
\[
a_i=\bigvee\Set{b_s}{\dom(s)=i}
\]
and consequently
\[
\bigvee\Set*{b_s}{s\in{^{<\omega}2}}=\bigvee\Set*{\bigvee\Set{b_s}{\dom(s)=i}}{i<\omega}=\bigvee\Set{a_i}{i<\omega}=\btop.
\]
Thirdly, we observe that for every $n<\omega$
\[
\Set*{s\in{^{<\omega}2}}{b_s\wedge x_n>\bbot\text{ and }b_s\wedge\neg x_n>\bbot}\subseteq\Set*{s\in{^{<\omega}2}}{\dom(s)\le n}
\]
and the set on the right-hand side is finite. In conclusion, if we enumerate
\[
\Set{p_i}{i<\omega}=\Set*{b_s}{s\in{^{<\omega}2}}\setminus\{\bbot\},
\]
then it is clear that $\Set{p_i}{i<\omega}$ is a maximal antichain in $\B$ satisfying the two desired properties.

For every $n<\omega$, the set $X_n=\Set*{i<\omega}{p_i\wedge x_n>\bbot}$ is such that $\bigvee\Set*{p_i}{i\in X_n}\in U$ hence, by the assumption that $U$ is a coherent $P$-ultrafilter, we can find $Y\subseteq\omega$ such that $y=\bigvee\Set{p_i}{i\in Y}\in U$ and for each $n<\omega$, $Y\setminus X_n$ is finite. Now we observe that for $n<\omega$
\[
\Set{i<\omega}{p_i\wedge y\wedge\neg x_n>\bbot}=\Set{i\in Y}{p_i\wedge x_n>\bbot\text{ and }p_i\wedge\neg x_n>\bbot}\cup(Y\setminus X_n)
\]
which is finite. From this and the fact that $\Set{p_i}{i<\omega}$ refines $\Set{a_i}{i<\omega}$, it follows that for each $n<\omega$ also $\Set{i<\omega}{a_i\wedge y\wedge\neg x_n>\bbot}$ is finite.

$(\ref{lemma:chardue}\Longrightarrow\ref{lemma:charuno})$ Let $\Set{a_i}{i<\omega}$ be a maximal antichain in $\B$. Assume that for each $n<\omega$, $X_n\subseteq\omega$ is such that $x_n=\bigvee\Set{a_i}{i\in X_n}\in U$. By our hypothesis, there exists $y\in U$ such that for each $n<\omega$ the set $\Set{i<\omega}{a_i\wedge y\wedge\neg x_n>\bbot}$ is finite. Letting $Y=\Set{i<\omega}{a_i\wedge y>\bbot}$, it is obvious that $\bigvee\Set{a_i}{i\in Y}\in U$ and for each $n<\omega$, $Y\setminus X_n$ is finite.
\end{proof}

We note that general existence results for coherent $P$-ultrafilters are proved in Star\'y \cite[Section 3]{stary:coherent}; however, the following is all we need for the purpose of this section.

\begin{proposition}[{Star\'y \cite[Proposition 3.4]{stary:coherent}}]\label{proposition:cexist} Assume $\dd=2^{\aleph_0}$. Let $\B$ be a complete c.c.c.\ Boolean algebra of cardinality $2^{\aleph_0}$. Then there exists a non-principal coherent $P$-ultrafilter on $\B$.
\end{proposition}

An important observation is that, as opposed to $P$-points, coherent $P$-ultrafilters may be Tukey maximal. For example, on $\C_{2^{\aleph_0}}$ there exists a coherent $P$-ultrafilter by Proposition \ref{proposition:cexist}, which is necessarily Tukey maximal by Theorem \ref{theorem:tukeymax}.

\begin{lemma}\label{lemma:dense} Let $\B$ be a complete c.c.c.\ Boolean algebra, with a dense subalgebra $\D\subseteq\B$. Let $U$ be a coherent $P$-ultrafilter on $\B$ and $A$ be a maximal antichain with the property that $A\cap U=\emptyset$. Then for all $u\in U$ there exists $v\in U$ such that $v\le u$ and $\Set{a\wedge v}{a\in A}\subseteq\D$.
\end{lemma}
\begin{proof} Let $u\in U$ be arbitrary; by density of $\D$, there exists a maximal antichain $\Set{p_i}{i<\omega}\subset\D$ such that:
\begin{itemize}
\item for all $i<\omega$ there exists $a\in A$ such that $p_i\le a$;
\item for all $i<\omega$, either $p_i\le u$ or $p_i\wedge u=\bbot$.
\end{itemize}

Using the assumption $A\cap U=\emptyset$, it follows that for all $a\in A$
\[
\bigvee\Set{p_i}{a\wedge p_i=\bbot}\in U.
\]
Now, by the fact that
\[
\Set*{X\subseteq\omega}{\bigvee\Set{p_i}{i\in X}\in U}
\]
is a $P$-point over $\omega$, there exists $Y\subseteq\omega$, with $\bigvee\Set{p_i}{i\in Y}\in U$, such that for every $a\in A$ the set $\Set{i\in Y}{a\wedge p_i>\bbot}$ is finite.

We claim that the element defined as
\[
v=u\wedge\bigvee\Set{p_i}{i\in Y}
\]
has the desired properties. Clearly $v\in U$ and $v\le u$. Furthermore, for each $a\in A$ we have
\[
a\wedge v=\bigvee\Set{a\wedge u\wedge p_i}{i\in Y}=\bigvee\Set{p_i}{i\in Y\text{ and }p_i\le a\wedge u},
\]
but there are only finitely many $i\in Y$ such that $a\wedge p_i>\bbot$. This means that $a\wedge v$ is a finite (possibly empty) disjunction of elements of the subalgebra $\D$, and therefore $a\wedge v\in\D$. 
\end{proof}

\begin{theorem}\label{theorem:weight} Let $\B$ be a complete c.c.c.\ Boolean algebra. If $U$ is a coherent $P$-ultrafilter on $\B$, then
\[
\Bigl\langle\fin*{\pi(\B)^+},\subseteq\Bigr\rangle\nleq_{\mathrm{T}}\langle U,\ge\rangle
\]
\end{theorem}
\begin{proof} Let $X\subseteq U$ be such that $\pi(\B)<\abs{X}$; we aim to find some infinite $Y\subseteq X$ such that $\bigwedge Y\in U$. If $U$ is $\sigma$-complete then this is immediate, so let us assume $U$ is not $\sigma$-complete, which means there exists a maximal antichain $\Set{a_i}{i<\omega}$ such that for all $i<\omega$, $a_i\notin U$.

Let $\D\subseteq\B$ be a dense subalgebra such that $\abs{\D}=\pi(\B)$. Without loss of generality, we may assume that for all $u\in X$, $\Set{a_i\wedge u}{i<\omega}\subseteq\D$. Indeed, by Lemma \ref{lemma:dense}, for all $u\in X$ there exists $v_u\in U$ such that $v_u\le u$ and $\Set{a\wedge v_u}{a\in A}\subseteq\D$. In case we have chosen the same $v_u$ for infinitely many $u\in X$, we have the thesis already. On the other hand, if the map $u\mapsto v_u$ is finite-to-one, without loss of generality we may replace $X$ with $\Set{v_u}{u\in X}$, which also has cardinality greater than $\pi(\B)$.

Now, we claim that there is some $x\in X$ such that for all $n<\omega$
\[
X_n=\Set{u\in X}{(\forall i<n)(a_i\wedge x=a_i\wedge u)}
\]
is infinite. If not, then for each $x\in X$ we could find some $n_x<\omega$ such that the set
\begin{equation}\label{eq:finite}
\Set{u\in X}{(\forall i<n_x)(a_i\wedge x=a_i\wedge u)}
\end{equation}
is finite. Clearly there exists $X'\subseteq X$, with $\pi(\B)<\abs{X'}$, such that for all $x\in X'$, $n_x=m$ for some fixed $m<\omega$. Now, the function
\[
\begin{split}
X' &\longrightarrow\D^m \\
x &\longmapsto\Seq{a_i\wedge x}{i<m}
\end{split}
\]
cannot be injective, hence there must be some infinite $X''\subseteq X'$ such that $x,x'\in X''$ implies $(\forall i<m)(a_i\wedge x=a_i\wedge x')$, contradicting the assumption that the set in \eqref{eq:finite} is finite. This completes the proof of the claim.

Since for every $n<\omega$ the set $X_n$ is infinite, recursively we can choose $\Set{x_n}{n<\omega}\subset X$ such that for each $n<\omega$, $x_n\in X_n\setminus\{x,x_0,\dots,x_{n-1}\}$. By Lemma \ref{lemma:char}, there exists $y\in U$ such that for each $n<\omega$ the set $\Set{i<\omega}{a_i\wedge y\wedge\neg x_n>\bbot}$ is finite. From now on, we follow the steps of Kanamori \cite[Theorem 1.10]{kanamori:ult}. For $n<\omega$ we define two natural numbers
\[
\gamma_n=\max\Set{\gamma<\omega}{(\forall i<\gamma)(a_i\wedge x=a_i\wedge x_n)}
\]
and
\[
\delta_n=\min\Set{\delta<\omega}{\gamma_n\le\delta\text{ and }(\forall i\ge\delta) (a_i\wedge y\le x_n)}.
\]
We also define the corresponding interval of natural numbers
\[
I_n=[\gamma_n,\delta_n),
\]
possibly empty.

Note that for each $n<\omega$ we have $n\le\gamma_n$, hence we can easily find by recursion an infinite set $W\subseteq\omega$ such that for all $n,m\in W$, if $I_n\cap I_m\neq\emptyset$ then $n=m$. Now there is clearly an infinite subset $Z\subseteq W$ with the property that
\[
z=\bigvee\Set*{a_i}{i\in\bigcup_{n\in Z}I_n}\notin U.
\]
Since $U$ is an ultrafilter, we have $x\wedge y\wedge\neg z\in U$; our goal is now to show that for every $n\in Z$
\begin{equation}\label{eq:final}
x\wedge y\wedge\neg z\le x_n.
\end{equation}
Let $n\in Z$ be given. As $\Set{a_i}{i<\omega}$ is a maximal antichain, to establish \eqref{eq:final} it is sufficient to prove that for any $i<\omega$
\[
a_i\wedge x\wedge y\wedge\neg z\le x_n.
\]
If $i$ is such that $a_i\le z$ then we are done, so let us assume $a_i\wedge\neg z>\bbot$. This implies that $i\notin I_n$, so there are now two possibilities. If $i<\gamma_n$ then by definition $a_i\wedge x\le x_n$. Otherwise, if $\delta_n\le i$ then again by definition $a_i\wedge y\le x_n$. This establishes \eqref{eq:final} for each $n\in Z$.

In conclusion, by taking $Y=\Set{x_n}{n\in Z}$, it follows that $Y$ is an infinite subset of $X$ such that $x\wedge y\wedge\neg z\le\bigwedge Y$ and therefore $\bigwedge Y\in U$. The thesis now follows from Proposition \ref{proposition:criterion}.
\end{proof}

\begin{corollary} Assume $\dd=2^{\aleph_0}$. Let $\B$ be a complete c.c.c.\ Boolean algebra of cardinality $2^{\aleph_0}$. If $\B$ has a dense subalgebra of cardinality $<2^{\aleph_0}$, then there exists a non-principal ultrafilter on $\B$ which is not Tukey maximal.
\end{corollary}
\begin{proof} By Proposition \ref{proposition:cexist} and Theorem \ref{theorem:weight}.
\end{proof}

The corollary gives that, assuming $\dd=2^{\aleph_0}$, for every cardinal $\aleph_0\le\kappa<2^{\aleph_0}$ there exists an ultrafilter on $\C_\kappa$ which is not Tukey maximal. However, this result does not apply to $\B_\omega$, because $\dd=2^{\aleph_0}$ implies that all dense subalgebras of $\B_\omega$ have cardinality $2^{\aleph_0}$ (by the fact that $\dd\le\cof(\N)$ and \cite[Theorem 2.5]{ckp:dense}). Therefore, in the next section we shall take a different approach.

\section{Non-maximal ultrafilters on the random algebra}\label{section:cinque}

This section complements the previous one by providing the construction of an ultrafilter on $\B_\omega$ which is not Tukey maximal, assuming Jensen's $\Diamond$.

\begin{definition}[Jensen \cite{jensen:fine}] Let $\Diamond$ be the following principle: there exists a sequence $\Seq{S_\alpha}{\alpha<\omega_1}$ such that $S_\alpha\subseteq\alpha$ and, for every $X\subseteq\omega_1$, the set $\Set{\alpha<\omega_1}{X\cap\alpha=S_\alpha}$ is stationary.
\end{definition}

The above principle easily implies $2^{\aleph_0}=\aleph_1$ and, by \cite[Lemma 6.5]{jensen:fine}, is true in the constructible universe $L$.

The first construction of non-maximal ultrafilters over $\omega$ using diamond principles is due to Milovich \cite[Theorem 3.11]{milovich:tukey}. Here we focus on $\B_\omega$, which for the sake of convenience we identify with $\mathcal{B}(^\omega2)\setminus\N$ during the proof of the next theorem. More explicitly, we shall not distinguish between a set $X\in\mathcal{B}(^\omega2)$ and its equivalence class $\eq{X}{\N}$, with the obvious convention that $\mu(X)=\mu\bigl(\eq{X}{\N}\bigr)$.

\begin{theorem} Assuming $\Diamond$, there exists an ultrafilter on $\B_\omega$ which is not Tukey maximal.
\end{theorem}
\begin{proof} By $\Diamond$, let us fix a sequence $\Seq{S_\alpha}{\alpha<\omega_1}$ such that $S_\alpha\subseteq\alpha$ and, for every $X\subseteq\omega_1$, the set $\Set{\alpha<\omega_1}{X\cap\alpha=S_\alpha}$ is stationary. As noted before, the existence of such a sequence implies $2^{\aleph_0}=\aleph_1$, hence we may enumerate $\B_\omega=\Set{B_\alpha}{\alpha<\omega_1}$. We proceed to construct recursively a sequence $\Seq{U_\alpha}{\alpha<\omega_1}$ such that, at each stage $\alpha<\omega_1$, the set $\Set{U_\beta}{\beta\le\alpha}\subset\B_\omega$ has the finite intersection property.

Let $U_0=\btop$. For $\alpha<\omega_1$, inductively the set $\Set{U_\beta}{\beta\le\alpha}$ has the finite intersection property, so we may choose $U_{\alpha+1}$ to be either $B_\alpha$ or $\neg B_\alpha$ in such a way that $\Set{U_\beta}{\beta\le\alpha+1}$ still has the finite intersection property. Suppose now $\delta<\omega_1$ is a limit ordinal; we distinguish two cases:
\begin{enumerate}
\item\label{caseone} There exists an infinite $Y\subseteq S_\delta$ such that $\Set{U_\alpha}{\alpha<\delta}\cup\bigl\{\bigcap_{\alpha\in Y}U_\alpha\bigr\}$ has the finite intersection property. Then choose such a $Y$ and define $U_\delta=\bigcap_{\alpha\in Y}U_\alpha$.
\item There is no such a $Y$. Then define $U_\delta=\btop$.
\end{enumerate}
This completes the recursive construction, whence it follows that $U=\Set{U_\alpha}{\alpha<\omega_1}$ is an ultrafilter on $\B_\omega$.

Now let $X\subseteq\omega_1$ be uncountable; we aim to find an infinite $Y\subseteq X$ such that $\bigcap_{\alpha\in Y}U_\alpha\in U$, then conclude by Proposition \ref{proposition:criterion} that $U$ is not Tukey maximal. By stationarity there exists, for some sufficiently large $\kappa$, a countable elementary substructure $M\preceq H_\kappa$ such that $\{X,U\}\subset M$, and a limit ordinal $\delta<\omega_1$ such that $X\cap M=S_\delta$. Now we only need to find an infinite $Y\subseteq X\cap M$ such that $\Set{U_\alpha}{\alpha<\delta}\cup\bigl\{\bigcap_{\alpha\in Y}U_\alpha\bigr\}$ has the finite intersection property, as this will imply that, at stage $\delta$, we are in Case \eqref{caseone} of the construction of $U$.

To simplify notation, as $\delta$ is countable, we may enumerate all the possible finite intersections from $\Set{U_\alpha}{\alpha<\delta}$ as $\Set{F_m}{m<\omega}$. By recursion, we shall construct for all $n<\omega$ an uncountable $X_n\in\mathcal{P}(X)\cap M$, a countable ordinal $\alpha_n\in X_n\cap M$, and non-empty clopen sets $\Seq{C_{n,m}}{m\le n}$. At each step $n$, the following crucial property will hold true:
\begin{equation}\label{eq:club}\tag{$\star_n$}
(\forall m\le n)(\forall\alpha\in X_n)\ \mu\Biggl(F_m\cap\bigcap_{k<n} U_{\alpha_k}\cap U_\alpha\cap C_{n,m}\Biggr)\ge\biggl(1-\frac{1}{2^{n+2}}\biggr)\mu(C_{n,m}).
\end{equation}
Once the recursive construction is complete, we shall conclude the proof by showing that $Y=\Set{\alpha_n}{n<\omega}$ has the desired property.

For the base case of the recursion, we first observe that for every $\alpha\in X$ the set $F_0\cap U_\alpha$ has positive measure. But $X$ is uncountable, hence by Proposition \ref{proposition:dense} there exists an uncountable $X_0\subseteq X$ and a non-empty clopen $C_{0,0}$ such that
\begin{equation}\tag{$\star_0$}
(\forall\alpha\in X_0)\ \mu(F_0\cap U_\alpha\cap C_{0,0})\ge\frac{3}{4}\mu(C_{0,0}),
\end{equation}
as we wanted. Now choose any $\alpha_0\in X_0$. At this point, we observe that the construction which has just been described can be performed within the model $M$, hence we can further assume that $X_0\in M$ and $\alpha_0\in M$.

For the general case we proceed similarly, but taking into account the decreasing measure of the sets $C_{n,m}$. More precisely, let $n>0$ and suppose inductively we have the countable ordinals $\Set{\alpha_k}{k<n}$, an uncountable $X_{n-1}\subseteq X$, and clopen sets $\Seq{C_{n-1,m}}{m<n}$. As in particular $\alpha_{n-1}\in X_{n-1}$, by $(\star_{n-1})$ we have for all $m<n$
\begin{equation}\label{eq:d}
\mu\Biggl(F_m\cap\bigcap_{k<n} U_{\alpha_k}\cap C_{n-1,m}\Biggr)\ge\biggl(1-\frac{1}{2^{n+1}}\biggr)\mu(C_{n-1,m}).
\end{equation}
Therefore, if $m<n$ then for every $\alpha\in X_{n-1}$
\begin{multline*}
\mu\Biggl(F_m\cap\bigcap_{k<n} U_{\alpha_k}\cap U_\alpha\cap C_{n-1,m}\Biggr)\\
\ge\mu\Biggl(F_m\cap\bigcap_{k<n} U_{\alpha_k}\cap C_{n-1,m}\Biggr)+\mu\Biggl(F_m\cap\bigcap_{k<n-1} U_{\alpha_k}\cap U_\alpha\cap C_{n-1,m}\Biggr)-\mu(C_{n-1,m})\\
\ge\biggl(1-\frac{1}{2^{n+1}}\biggr)\mu(C_{n-1,m})+\biggl(1-\frac{1}{2^{n+1}}\biggr)\mu(C_{n-1,m})-\mu(C_{n-1,m})\\
=\biggl(1-\frac{1}{2^n}\biggr)\mu(C_{n-1,m}),
\end{multline*}
where the second line follows from finite additivity of the measure $\mu$, whereas the third is a consequence of \eqref{eq:d} and $(\star_{n-1})$. Now, since $X_{n-1}$ is uncountable, by Proposition \ref{proposition:dense} again there exists an uncountable $X_n\subseteq X_{n-1}$ and non-empty clopens $\Seq{C_{n,m}}{m\le n}$ such that \eqref{eq:club} holds and furthermore:
\[
\text{if }m<n\text{, then }C_{n,m}\subseteq C_{n-1,m}\text{ and }\mu(C_{n,m})\ge\biggl(1-\frac{1}{2^n}\biggr)\mu(C_{n-1,m}).
\]
Now choose any $\alpha_n\in X_n\setminus\{\alpha_0,\dots,\alpha_{n-1}\}$. Again, we observe that each step of the recursive construction can be performed within the model $M$, hence we can further assume that $X_n\in M$ and $\alpha_n\in M$.

We have thus obtained $\Set{\alpha_n}{n<\omega}\subseteq X\cap M$, so the proof is complete once we show that, for each fixed $m<\omega$,
\[
\mu\Biggl(F_m\cap\bigcap_{n<\omega}U_{\alpha_n}\Biggr)>0.
\]
To do so, consider $\varepsilon=\prod_{\ell=1}^{\infty}\bigl(1 - \frac{1}{2^\ell}\bigr)>0$ and observe that, by our choice of the clopen sets, for every $n\ge m$ we have
\[
\mu(C_{n,m})\ge\prod_{\ell=1}^{n}\biggl(1 - \frac{1}{2^\ell}\biggr)\mu(C_{m,m})\ge\varepsilon\,\mu(C_{m,m}).
\]
Consequently, for every $n\ge m$
\[
\mu\Biggl(F_m\cap\bigcap_{k<n} U_{\alpha_k}\Biggr)\ge\frac{3}{4}\mu(C_{n,m})\ge\varepsilon\frac{3}{4}\mu(C_{m,m})
\]
and finally, by countable additivity of the measure $\mu$,
\[
\mu\Biggl(F_m\cap\bigcap_{n<\omega}U_{\alpha_n}\Biggr)=\lim_{n\to\infty}\mu\Biggl(F_m\cap\bigcap_{k<n} U_{\alpha_k}\Biggr)\ge\varepsilon\frac{3}{4}\mu(C_{m,m})>0,
\]
as desired. This shows that $Y=\Set{\alpha_n}{n<\omega}$ has the required properties and completes the proof.
\end{proof}

\begin{question} Does $2^{\aleph_0}=\aleph_1$ imply the existence of an ultrafilter on $\B_\omega$ which is not Tukey maximal?
\end{question}

\end{document}